\documentclass[11pt]{amsart}
\usepackage{a4wide,amsfonts,color}
\usepackage{amssymb}
\usepackage{amsthm}
\usepackage{newlfont}
\usepackage[all]{xy}

\newtheorem{theorem}{Theorem}

\newtheorem{lemma}{Lemma}
\newtheorem{proposition}{Proposition}
\newtheorem{corollary}{Corollary}
\newtheorem{claim}[theorem]{Claim}
\theoremstyle{definition}
\newtheorem{definition}{Definition}
\newtheorem{example}{Example}

\newtheorem{problem}{Problem}

\newcommand{\C}{\mathcal C}
\newcommand{\M}{\mathcal M}

\newcommand{\I}{\mathcal I}

\newcommand{\F}{\mathcal F}

\newcommand{\A}{\mathcal A}
\newcommand{\cB}{\mathcal B}

\newcommand{\Tau}{\mathcal T}

\newcommand{\pr}{\mathrm{pr}}

\newcommand{\e}{\varepsilon}

\newcommand{\IR}{\mathbb R}
\newcommand{\IT}{\mathbb T}

\newcommand{\IZ}{\mathbb Z}

\newcommand{\IN}{\mathbb N}
\newcommand{\w}{\omega}
\newcommand{\Ra}{\Rightarrow}

\newcommand{\diam}{\mathrm{ diam}\,}

\newcommand{\cov}{\mathrm{cov}}

\newcommand{\Borel}{\mathcal{B}or}

\title[A large Borel linear subspace in the countable product of lines]{A Borel linear subspace of $\mathbb R^\w$ that cannot be covered by countably many closed Haar-meager  sets}
\author{Taras Banakh, Eliza Jab{\l}o\'{n}ska}
\address{T.~Banakh: Jan Kochanowski University in Kielce (Poland), and
Ivan Franko National University in Lviv (Ukraine)}
\email{t.o.banakh@gmail.com}
\address{E.~Jab{\l}o\'{n}ska: AGH University of Science and Technology, Faculty of Applied Mathematics, Mickiewicza 30, 30--059 Krak{\'o}w, Poland}
\email{elizajab@agh.edu.pl}
\subjclass{28C10; 46A03; 54C30; 54H05}

\begin{document}

\begin{abstract} We prove that the countable product of lines contains a Borel linear subspace $L\ne\IR^\w$ that cannot be covered by countably many closed Haar-meager sets. This example is applied to studying the interplay between various classes of ``large'' sets and Kuczma--Ger classes in the topological vector spaces $\IR^n$ for $n\le \w$. 
\end{abstract}
\maketitle

\section{Introduction} 

By the classical Steinhaus Theorem \cite{Stein}, for any Borel subset $A$ of positive Lebesgue measure on the real line, the difference $A-A=\{x-y:x,y\in A\}$ is a neighborhood of zero. In \cite{Weil} Weil extended this result of Steinhaus to all locally compact topological groups proving that for any Borel subset $A$ of positive Haar measure in a locally compact topological group $X$, the set $AA^{-1}=\{xy^{-1}:x,y\in A\}$ is a neighborhood of the identity of the group $X$. This result implies that any nonopen Borel subgroup of a locally compact topological group $X$ belongs to the $\sigma$-ideal $\mathcal N_X$ of subsets of Haar measure zero in $X$.

A Baire category analogue of the Steinhaus--Weil Theorem was obtained by Ostrowski \cite{Ostrowski}, Piccard \cite{Piccard} and Pettis \cite{Pettis} who proved that for any nonmeager Borel subset $A$ of a Polish topological group $X$, the set $AA^{-1}$ is a neighborhood of the identity. This implies that any nonopen Borel subgroup in a Polish group $X$ is meager in $X$ and hence belongs to the $\sigma$-ideal $\mathcal M_X$ of meager subsets of $X$.

Therefore, any nonopen Borel subgroup of a locally compact Polish group $X$ belongs to the $\sigma$-ideal $\mathcal M_X\cap\mathcal N_X$ of meager subsets of Haar measure zero in $X$. For the real line this result was essentially improved by Laczkovich \cite{La} who proved that every nonopen analytic subgroup of the real line can be covered by countably many closed sets of Lebesgue measure zero. In \cite{BBJ} this result of Laczkovich was generalized as follows.

\begin{theorem}\label{t:BBJ} Any nonopen analytic subgroup of a locally compact topological group $X$ can be covered by countably many closed sets of Haar measure zero in $X$.
\end{theorem}
We recall that a topological space  is {\em analytic} if it is a continuous image of a Polish space.

It is well-known that a Polish group admits a non-trivial left-invariant Borel $\sigma$-additive measure if and only if it is locally compact. Nonetheless a counterpart of the ideal $\mathcal N_X$ of Haar-null sets can be defined in any Polish (not necessarily locally compact) group. Following Christensen \cite{Christensen}, we define a Borel subset $A$ of a topological group $X$ to be {\em Haar-null} if  there exists a probability measure $\mu$ on $X$ such that $\mu(xAy)=0$ for all $x,y\in X$. By \cite[Theorems 3.2.5 and 3.3.7]{EN}, for any Polish group $X$, the family $\mathcal{HN}_{\!X}$ of subsets of Borel Haar-null sets of $X$ is a $\sigma$-ideal, which coincides with the ideal $\mathcal N_X$ of sets of Haar measure zero if the group $X$ is locally compact. By \cite[Theorem 4.3]{BGJS}, a Borel susbet $A$ of a Polish Abelian group $X$ is Haar-null if and only if there exists a continuous map $f:2^\w\to X$ from the Cantor cube $2^\w=\{0,1\}^\w$ such that for every $x\in X$ the preimage $f^{-1}[A+x]$ has measure zero in the Cantor cube $2^\w$ endowed with the standard product measure. 

This characterization of Haar-null sets suggests to introduce a  Baire category analog of Haar-null sets, replacing the ideal of null sets in $2^\w$ by the ideal of meager sets in $2^\w$. This was done by Darji \cite{Darji} who defined a Borel subset $A$ of a Polish group $X$ to be {\em Haar-meager} if there exists a continuous map $f:2^\w\to X$ such that for every $x,y\in X$ the preimage $f^{-1}[xAy]$ is a meager subset of $X$.  By \cite[Theorems 3.2.6 and 3.3.13]{EN}, for any Polish group $X$, the  family $\mathcal{HM}_X$ of all subsets of Haar-meager Borel sets of $X$ is a $\sigma$-ideal, which coincides with the ideal $\mathcal M_X$ of meager subsets of $X$ if the group $X$ is locally compact. It is easy to see that each closed Haar-null set is Haar-meager, which implies that the $\sigma$-ideal $\sigma\overline{\mathcal{HN}}_{\!X}$ generated by closed Haar-null sets in $X$ is a subideal of the $\sigma$-ideal $\sigma{\overline{\mathcal{HM}}}_X$ generated by closed Haar-meager sets in $X$. For any Polish group $X$ we have the inclusions
$$\sigma\overline{\mathcal{HN}}_{\!X}\subseteq \mathcal{HN}_{\!X}\cap \sigma\overline{\mathcal{HM}}_X
\subseteq  \sigma\overline{\mathcal{HM}}_{\!X}\subseteq \mathcal{HM}_X\subseteq\mathcal{M}_X.$$If the group $X$ is locally compact, then two last inclusions turn into equalities:
$$\sigma\overline{\mathcal{HN}}_{\!X}\subseteq \mathcal{HN}_{\!X}\cap \sigma\overline{\mathcal{HM}}_X
\subseteq  \sigma\overline{\mathcal{HM}}_X=\mathcal{HM}_X=\mathcal{M}_X.$$
On the other hand, the Polish Abelian group $\IZ^\w$ contains a Borel subgroup $H$  which is Haar-null and Haar-meager in $\IZ^\w$ but cannot be covered by countably many closed Haar-meager sets, see \cite[Example 8.1]{BGJS}. This example shows that Theorem~\ref{t:BBJ} does not extend to non-locally compact groups. In this paper we prove that an example distinguishing the $\sigma$-ideals $\sigma\overline{\mathcal{HM}}_X$ and $\mathcal{HM}_X$ can be found among Borel linear subspaces of the countable product of lines.

The main result of this paper is the following 

\begin{theorem}\label{t:main} The countable product of lines $\mathbb R^\w$ contains a  Borel linear subspace, which is Haar-meager and Haar-null in $\IR^\w$ but cannot be covered by countably many closed Haar-meager sets in $\IR^\w$.
\end{theorem}

Theorem~\ref{t:main} will be proved in Section~\ref{s:pf}. Now we apply this theorem for better understanding the borderline between various classes of ``large'' sets in Polish vector spaces and  Kuczma--Ger classes $\A_X,\cB_X,\C_X$. Those classes are related to the problem of automatic continuity of additive or mid-convex functions on Polish vector spaces. All vector spaces considered in this paper are over the field $\IR$ of real numbers. By a {\em Polish vector space} we understand a nontrivial topological vector space homeomorphic to a separable complete metric space.

 A function $f:X\to \IR$ on a subset $X$ of a vector space is called 
\begin{itemize}
\item {\em additive} if $f(x+y)=f(x)+f(y)$ for every $x,y\in X$ with $x+y\in X$;
\item  {\em mid-convex} if $f\big(\frac{x+y}2\big)\le\frac{f(x)+f(y)}2$ for every $x,y\in X$ with $\frac{x+y}2\in X$.
\end{itemize}
It is easy to see that an additive function $f:X\to\IR$ on a topological vector space $X$ is continuous if and only if it is (upper) bounded on some nonempty open set. Consequently, an additive function $f:X\to\IR$ is continuous if $f$ is upper bounded on some set $T\subseteq X$ such that the set $T-T$ has a non-empty interior.

This observation motivated Kuczma and Ger \cite{GKu} to introduce the following classes of sets in each topological vector space $X$:
\begin{itemize}
\item $\A_X$ is the class of all sets $A\subseteq X$ such that any mid-convex function $f:D\to\IR$ defined on an open convex subset $D\supseteq A$ is continuous whenever $\sup f[A]<\infty$;
\item $\mathcal B_X$ is the class of all sets $B\subseteq X$ such that any additive function $f:X\to\IR$ is continuous whenever $\sup f[B]<\infty$;
\item $\mathcal C_X$ is the class of all sets $C\subseteq X$ such that any additive function $f:X\to\IR$ is continuous whenever $f[C]$ is a bounded subset of $\IR$.
\end{itemize}
It is clear that $\mathcal A_X\subseteq\mathcal B_X\subseteq\mathcal C_X$. In \cite{GKo} Ger and Kominek proved that $\mathcal A_X=\mathcal B_X$ for every Baire topological vector space $X$. On the other hand $\mathcal B_X\ne\mathcal C_X$ for any topological vector space $X$. Indeed, for any discontinuous additive function $f:X\to \IR$ the set  $T=\{x\in X:f(x)\le 0\}$ belongs to the class $\C_X$ but not to the class $\mathcal B_X$, see \cite{Erdos}. It should be noted that the Axiom of Choice implies the existence of $2^{|X|}$ discontinuous additive functions on each topological vector space $X$. 

Now, in addition to the families $\M_X$, $\mathcal{HN}_{\!X}$, $\mathcal{HM}_X$, $\sigma\overline{\mathcal{HN}}_{\!X}$, $\sigma\overline{\mathcal{HM}}_X$ we recall definitions of some other families of sets in topological groups that are related to the Kuczma--Ger classes. 

A subset $T$ of the Cantor cube $2^{\omega}$ is called {\em thin} if for every number $n\in\omega$ the restriction $\mbox{pr}|_T$ of the projection $\mbox{pr}:2^{\omega}\to 2^{\omega\setminus\{n\}}$, $\mathrm{pr}:x\mapsto x{\restriction}_{\omega\setminus \{n\}}$, is injective.

\begin{definition}
A subset $A$ of a Polish Abelian group $X$ is called:
\begin{itemize}
\item {\em  Haar-thin} if there exists a continuous  map $f:2^{\omega}\to X$ such that $f^{-1}[A+x]$ is thin for every $x\in X,$ 
\item {\em null-finite} if there is a null sequence $\{x_n\}_{n\in\w}\subseteq X$ such that the set
$\{n\in\omega:x_n+x\in A\}\;\mbox{is finite for every}\;x\in X.$
\end{itemize}
\end{definition}

Given a Polish Abelian group $X$, denote by $\mathcal{HT}_{\!X}$ and $\mathcal{NF}_{\!X}$ the families of Haar-thin and null-finite sets in $X$, respectively. Also, let $\mathcal P_X$ be the family of all subsets of $X$, and $\Borel_X$ be the family of all Borel  subsets in $X$. 

By \cite[\S 9]{BGJS}, each Borel Haar-thin subset of a Polish Abelian group $X$ is Haar-null and Haar-meager. Therefore, 
$$\Borel_X\cap\mathcal{HT}_{\!X}\subseteq\mathcal{HN}_{\!X}\cap\mathcal{HM}_X.$$ By \cite[Theorem 9.7]{BGJS}, a subset $A$ of a Polish Abelian group $X$ is Haar-thin if and only if $A-A$ is not a neighborhood of zero. This implies the following 

\begin{corollary} If a subset of a Polish vector space $X$ is not Haar-thin, then it belongs to the Kuczma--Ger class $\mathcal C_X$.
\end{corollary}

Therefore, $\mathcal P_X\setminus\mathcal{HT}_{\!X}\subseteq\mathcal C_X$. On the other hand, for any  Polish vector space $X$ and any discontinuous additive function $f:X\to\IR$ the set $T=\{x\in X:f(x)\ge 0\}$ does not belong to the Kuczma--Ger class $\mathcal B_X$ but is not thin (as $T-T=X$ is a neighborhood of zero). This example shows that $\mathcal P_X\setminus\mathcal{HT}_{\!X}\not\subseteq\mathcal B_X$ for any Polish vector space $X$ (by our assumption, all vector spaces are non-trivial).

\begin{problem} Is $\Borel_X\setminus\mathcal{HT}_{\!X}\subseteq \mathcal B_X$ for some  Polish vector space $X$?
\end{problem}

By \cite[Theorems 5.1 and 6.1]{BJ}, each Borel null-finite subset of a Polish Abelian group $X$ is Haar-null and Haar-meager. Therefore, $\Borel_X\cap\mathcal{NF}_{\!X}\subseteq \mathcal{HM}_X\cap\mathcal{HM}_X$. By \cite[Theorem 9.1]{BJ}, a subset $A$ of a Polish vector space $X$ belongs to the Kuczma--Ger class $\mathcal B_X$ if $A$ is not null-finite in $X$. Therefore, $\mathcal P_X\setminus\mathcal{NF}_{\!X}\subseteq\mathcal B_X$ for any Polish vector space $X$.

Now we turn to the $\sigma$-ideals $\sigma\overline{\mathcal{HN}}_X$ and $\sigma\overline{\mathcal{HM}}_X$. By \cite[Corollary 5]{BBJ}, any analytic set $A\notin\sigma\overline{\mathcal{HN}}_X$ in a  finite-dimensional topological vector space $X$ belongs to the Kuczma--Ger class $\mathcal B_X$. The situation changes for infinite-dimensional Polish vector spaces.

\begin{theorem} The Polish vector space $X=\IR^\w$ contains a linear subspace $L$ such that $$L\in\Borel_X\cap\mathcal{HT}_{\!X}\cap\mathcal{NF}_{\!X}\setminus\sigma{\overline{\mathcal{HM}}_X\subseteq\mathcal{HN}_{\!X}\cap\mathcal{HM}_X\setminus\sigma\overline{\mathcal{HN}}}_{\!X}$$but $L\notin\mathcal C_X$.
\end{theorem}

\begin{proof} Let $L$ be the Borel linear subspace of $\IR^\w$, constructed in Theorem~\ref{t:main}. Since $L\ne\IR^\w$, there exists a discontinuous additive function $f:\IR^\w\to\IR$ such that $f[L]=\{0\}$. By \cite[Theorem 9.1]{BJ}, the set $L$ is null-finite. Since $L=L-L$ is not open in $\IR^\w$, $L$ is Haar-thin in $\IR^\w$ by \cite[Theorem 9.7]{BGJS}. Therefore, $$L\in\Borel_X\cap\mathcal{HT}_{\!X}\cap\mathcal{NF}_{\!X}\setminus\sigma{\overline{\mathcal{HM}}_X\subseteq\mathcal{HN}_{\!X}\cap\mathcal{HM}_X\setminus\sigma\overline{\mathcal{HN}}}_{\!X}.$$
The discontinuous additive function $f$, which is bounded on the set $L$, witnesses that $L\notin\mathcal C_X$.
\end{proof}

The following corollary sums up the preceding discussion.

\begin{corollary} For any nonzero cardinal $n\le\w$ and the Polish vector space $X=\IR^n$, the following statements hold:
\begin{enumerate}
\item $\Borel_X\cap(\mathcal{HT}_{\!X}\cup\mathcal{NF}_{\!X})\subseteq\mathcal{HN}_{\!X}\cap\mathcal{HM}_X$;
\item $\sigma\overline{\mathcal{HN}}_{\!X}\subseteq\mathcal{HN}_{\!X}\cap\sigma\overline{\mathcal{HM}}_X\subseteq\sigma\overline{\mathcal{HM}}_X\subseteq\M_X$;
\item $\mathcal P_X\setminus\mathcal{HT}_{\!X}\subseteq\mathcal C_X$;
\item $\mathcal P_X\setminus\mathcal{NF}_{\!X}\subseteq\mathcal A_X=\mathcal B_X\subseteq\C_X$;
\item If $n<\w$, then $\Borel_X\setminus\sigma\overline{\mathcal{HN}}_{\!X}\subseteq\mathcal A_X=\mathcal B_X\subseteq\C_X$;
\item If $n=\w$, then $\Borel_X\cap\mathcal{HT}_{\!X}\cap\mathcal{NF}_{\!X}\setminus\sigma\overline{\mathcal{HM}}_X\not\subseteq\mathcal C_X$.
\end{enumerate}
\end{corollary}

\begin{problem} Does every infinite-dimensional Banach space $X$ contain a meager Borel linear subspace $H\notin\sigma\overline{\mathcal{HM}}_X$?
\end{problem}

\section{Proof of Theorem~\ref{t:main}}\label{s:pf}

In this section we present a proof of Theorem~\ref{t:main}, which uses the ideas of  the construction of a Borel subgroup $H\notin\sigma\overline{\mathcal{HM}}_{\IZ^\w}$ in Example 8.1 in \cite{BGJS}. First we recall  some  notions and facts from \cite{BGJS}.

\begin{definition} 
A subset $A$ of a Polish Abelian group $X$ is called:
\begin{itemize}
\item {\em thick} if for any compact subset $K\subseteq X$ there is $x\in X$ such that $K\subseteq x+ A$,
\item {\em Haar-open} in $X$ if for any
compact subset $K\subseteq X$ and point $p\in K$ there exists $x\in X$ such that $K\cap(A + x)$ is a neighborhood of $p$ in $K$.
\end{itemize}
\end{definition}

The following two propositions are proved in Proposition 7.4 and Theorem 7.5 of \cite{BGJS}, respectively.

\begin{proposition}
If a subset $A$ of a Polish Abelian group $X$ is Haar-open, then for any
compact set $K\subseteq X$ there exists a finite set $F\subseteq X$ such that $K\subseteq F + A$.
\end{proposition}

\begin{proposition}\label{Ho}
A closed subset $A$ of a Polish Abelian group $X$ is Haar-meager if and only if it is not Haar-open in $X$.
\end{proposition}

We denote by $\w$ the set of finite ordinals and by $\IN=\w\setminus\{0\}$ the set of positive integer numbers.  A family $\F$ of sets is called {\em disjoint} if $A\cap B=\emptyset$ for any distinct sets $A,B\in\F$. For a function $f:X\to Y$ and subset $A\subseteq X$ by $f[A]$ we denote the image $\{f(a):a\in A\}$ of $A$ under the map $f$.

Let $\cov(\M)$ denote the smallest cardinality of a cover of the Polish space $\w^\w
$ by nowhere dense subsets. It is clear that $\w_1\le\cov(\mathcal M)\le\mathfrak c$. It is known \cite{Blass} that $\cov(\M)=\mathfrak c$ under Martin's Axiom.

Theorem~\ref{t:main} follows immediately from the following (a bit more precise) statement. 

\begin{example}\label{ex1}
The countable product of lines $\IR^\w$ contains a linearly independent Polish subspace $P$ such that 
\begin{itemize}
\item $P$ cannot be covered by less than $\cov(\M)$ closed Haar-meager subsets of $\IR^\w$;
\item the linear hull of $P$ is a meager Borel subspace of $\IR^\w$.
\end{itemize}
\end{example}

\begin{proof} In the construction of the set $P$ we shall use the following lemma. 

\begin{lemma}\label{l1} There exists an infinite family $\Tau$ of thick subsets of $\IR$ and an increasing sequence $(\Xi_m)_{m\in\w}$ of positive integer numbers such that for any positive numbers $n\le m$,  real numbers $\lambda_1,\dots,\lambda_n\in[-m,m]\setminus\big[{-}\frac1m,\frac1m\big]$, pairwise distinct sets $T_1,\dots,T_n\in\Tau$ and points $x_i\in T_i$, $i\le n$, such that $\{x_1,\dots,x_n\}\not\subseteq [-\Xi_m,\Xi_m]$ we have $|\lambda_1x_1+\dots+\lambda_n x_n|\ge m+1$.
\end{lemma}

\begin{proof} For every $m\in\w$, let $\xi_m\in\w$ be the smallest number such that $$2^{2^{x}}-x\ge m\big(1+m+m^2(2^{2^{x-1}}+x)\big)\quad\mbox{for all $x\ge \xi_m$},$$ and put $\Xi_m=2^{2^{\xi_m}}+\xi_m$. Choose an infinite family $\A$ of pairwise disjoint infinite subsets of $\IN$ and for every $A\in\A$ consider the thick subset $$T_A:=\bigcup_{a\in A}\big[2^{2^a}-a,2^{2^a}+a\big]$$ of $\IR$. Taking into account that the families $\big(\big[2^{2^n}-n,2^{2^n}+n\big]\big)_{n=1}^\infty$ and $\mathcal A$ are disjoint, we conclude that so is the family $\Tau=(T_A)_{A\in\A}$.

We claim that the disjoint family $\Tau$ and the sequence $(\Xi_m)_{m\in\w}$ have the required property.
Take any positive numbers $n\le m$, real numbers $\lambda_1,\dots,\lambda_n\in[-m,m]\setminus \big[-\frac1m,\frac1m\big]$, pairwise distinct sets $T_1,\dots,T_n\in\Tau$ and points $x_1\in T_1,\dots,x_n\in T_n$ such that $\{x_1,\dots,x_n\}\not\subseteq [-\Xi_m,\Xi_m]$. For every $i\le n$ find an integer number $a_i$ such that $x_i=2^{2^{a_i}}+\e_i$ for some $\e_i\in[-a_i,a_i]$. Since the family $\A$ is disjoint and the sets $T_1,\dots,T_n$ are pairwise distinct, the points $a_1,\dots,a_n$ are pairwise distinct, too. Let $j$ be the unique number such that $a_j=\max\{a_i:1\le i\le n\}$. Taking into account that $\{x_1,\dots,x_n\}\not\subseteq [0,\Xi_m]=[0,2^{2^{\xi_m}}+\xi_m]$, we conclude that $2^{2^{a_j}}+a_j\ge 2^{2^{a_j}}+\e_j=x_j>2^{2^{\xi_m}}+\xi_m$ and hence $a_j>\xi_m$. The definition of the number $\xi_m$ guarantees that $2^{2^{a_j}}-a_j>m(1+m+m^2(2^{2^{a_j-1}}+a_j))$ and hence
\begin{multline*}
|\lambda_jx_j|\ge \tfrac1mx_j=\tfrac1m\big(2^{2^{a_j}}+\e_j\big)\ge \tfrac1m\big(2^{2^{a_j}}-a_j\big)\ge\\
1+m+m^2\big(2^{2^{a_j-1}}+a_j\big)\ge 1+m+ \sum_{i\ne j}|\lambda_i|\big(2^{2^{a_i}}+a_i\big)\ge 1+m+ \Big|\sum_{i\ne j}\lambda_ix_i\Big|\ge m+1.
\end{multline*}
 Thus $|\sum_{i=1}^n\lambda_ix_i|\ge m+1$.
\end{proof}

Now we are ready to start the construction of the $G_\delta$-set $P\subseteq\IR^\w$. This construction will be done by induction on the tree $\w^{<\w}=\bigcup_{n\in\w}\w^n$ consisting of finite sequences $s=(s_0,\dots,s_{n-1})\in\w^n$ of finite ordinals. 
For a sequence $s=(s_0,\dots,s_{n-1})\in \w^n$ and a~number $m\in\w$ by $s\hat{\;}m=(s_0,\dots,s_{n-1},m)\in \w^{n+1}$ we denote the {\em concatenation} of $s$ and $m$.

For an infinite sequence $s=(s_n)_{n\in\w}\in\w^\w$ and a natural number $l\in\w$ let $s{\restriction}_l=(s_0,\dots,s_{l-1})$ be the restriction of the function $s:\w\to\w$ to the subset $l=\{0,\dots,l-1\}$.
Observe that the topology of the space $\IR^\w$ is generated by the metric $$\rho(x,y)=\max_{n\in\w}\frac{|x(n)-y(n)|}{2^n},\;\;x,y\in\IR^\w,$$and hence for every $z\in\IR^\w$ and $n\in\w$ the set $U(z{\restriction}_n)=\{x\in\IR^\w:\max_{i\in n}|x(i)-z(i)|\le 2^{-n}\}$ coincides with the closed ball $B[z;2^{-n}]=\{x\in\IR^\w:\rho(x,z)\le 2^{-n}\}$ of radius $2^{-n}$ centered at~$z$.

Using Lemma~\ref{l1}, choose a sequence $(\Xi_m)_{m\in\w}$ of positive numbers and a sequence $(T_s)_{s\in\w^{<\w}}$ of thick sets in the real line $\IR$ such that for every positive integer number $m$, finite set $F\subseteq \w^{<\w}$ of cardinality $|F|\le m$, function $\lambda:F\to [-m,m]\setminus\big[-\tfrac1m,\tfrac1m\big]$, and numbers $z_s\in T_s$, $s\in F$, such that $\{z_s\}_{s\in F}\not\subseteq[-\Xi_m,\Xi_m]$ we have $\big|\sum_{s\in F}\lambda(s)\cdot z_s\big|\ge m+1$.

For every $s\in\w^{<\w}$ and $n\in\w$ choose any point $t_{s,n}\in T_s\setminus[-n,n]$ and observe that the set $T_{s,n}=T_s\setminus(\{t_{s,n}\}\cup[-n,n])$ remains thick in $\IR$.

By induction on the tree $\w^{<\w}$ we shall construct a sequence $(z_s)_{s\in\w^{<\w}}$ of points of $\IR^\w$ and a sequence $(l_s)_{s\in\w^{<\w}}$ of finite ordinals satisfying the following conditions for every $s\in \w^{<\w}$:
\begin{itemize}
\item[$(1_s)$] $U(z_{s\hat{\;} i}{\restriction}_{l_{s\hat{\;} i}})\cap U(z_{s\hat{\;} j}{\restriction}_{l_{s\hat{\;} j}})=\emptyset$ for any distinct numbers $i,j\in\w$;
\item[$(2_s)$] $l_{s\hat{\;} i}> l_s+i$ for every $i\in\w$;
\item[$(3_s)$] the closure of the set $\{z_{s\hat{\;}i}\}_{i\in\w}$ contains the Haar-open set\\ $\mathbb T_s=U(z_s{\restriction}_{l_s})\cap(\IR^{l_s}\times\prod_{n\ge l_s}T_{s,n})$ and is contained in the set\\ $\hat{\mathbb T}_s=U(z_s{\restriction}_{l_s})\cap\big(\IR^{l_s}\times \prod_{n\ge l_s}(T_{s,n}\cup\{t_{s,n}\})\big)\subseteq U(z_s{\restriction}_{l_s})$.
\end{itemize}
We start the inductive construction letting $z_0=0$ and $l_0=0$. Assume that for some $s\in\w^{<\w}$, a point $z_s\in\IR^\w$ and a number $l_s\in\w$ have been constructed. Consider the Haar-open closed sets $\mathbb T_s$ and $\hat\IT_s$ defined in the condition $(3_s)$. Since $\mathbb T_s$ is nowhere dense in $\hat{\mathbb T}_s$, we can find a sequence $(z_{s\hat{\;}i})_{i\in\w}$ of pairwise distinct points of $\hat{\mathbb T}_s$ such that the space $D_s=\{z_{s\hat{\;}i}\}_{i\in\w}$ is discrete and contains $\mathbb T_s$ in its closure. Since $D_s$ is discrete, for every $i\in\w$ we can choose a number $l_{s\hat{\;}i}>l_s+i$ such that the balls $U(z_{s\hat{\;}i}{\restriction}_{l_{s\hat{\;}i}})$, $i\in\w$, are pairwise disjoint. Observing that the sequences $(z_{s\hat{\;}i})_{i\in\w}$ and $(l_{s\hat{\;}i})_{i\in\w}$ satisfy the conditions $(1_s)$--$(3_s)$, we complete the inductive step.

We claim that the $G_\delta$-subset $$P=\bigcap_{n\in\w}\bigcup_{s\in\w^n}U(z_s{\restriction}_{l_s})=\bigcup_{s\in\w^\w}\bigcap_{n\in\w}U(z_{s{\restriction}n}{\restriction}_{l_{s{\restriction}n}})$$ of $\IR^\w$ has required properties. First observe that the map $h:\w^\w\to P$ assigning to each infinite sequence $s\in\w^\w$ the unique point $z_s$ of the intersection $\bigcap_{n\in\w}U(z_{s{\restriction}n}{\restriction}_{l_{s{\restriction}n}})$ is a~homeomorphism of $\w^\w$ onto $P$. Then the inverse map $h^{-1}:P\to\w^\w$ is a~homeomorphism too.

\begin{claim}\label{cl:Ho} For every nonempty open subset $U\subseteq P$ its closure $\overline{U}$ in $\IR^\w$ is Haar-open in $\IR^\w$.
\end{claim}

\begin{proof} Pick any point $p\in U$ and find a unique infinite sequence $t\in\w^\w$ such that $\{p\}=\bigcap_{m\in\w}U(z_{t{\restriction}m}{\restriction}_{l_{t{\restriction}m}})$. Since the family $\{U(z_{t{\restriction}m}{\restriction}_{l_{t{\restriction}m}})\}_{m\in\w}$ is a neighborhood base at $p$, there is $m\in\w$ such that $P\cap U(z_{t{\restriction}m}{\restriction}_{l_{t{\restriction}m}})\subseteq U$. Consider the finite sequence $s=t{\restriction}_m$. By the definition of $P$, for every $i\in\w$ the intersection $P\cap U(z_{s\hat{\;}i}{\restriction}_{l_{s\hat{\;}i}})$ contains some point $y_{s\hat{\;}i}$. Taking into account that $$\rho(z_{s\hat{\;}i},y_{s\hat{\;}i})\le\diam\, U(z_{s\hat{\;}i}{\restriction}_{l_{s\hat{\;}i}})\le 2^{1-l_{s\hat{\;}i}}\le 2^{1-i}$$ and $\mathbb T_s$ is contained in the closure of the set $\{z_{s\hat{\;}i}\}_{i\in\w}$, we conclude that the Haar-open set $\IT_s$ is contained in the closure of the set $\{y_{s\hat{\;}i}\}_{i\in\w}\subseteq P\cap U(z_s{\restriction}_{l_s})\subseteq U$, which implies that $\overline{U}$ is Haar-open.
\end{proof}

\begin{claim} The set $P$
cannot be covered by less than $\cov(\M)$ many closed Haar-meager sets in $\IR^\w$.
\end{claim}

\begin{proof} To derive a contradiction, assume that $\mathcal H$ is a family of closed Haar-meager subsets of $\IR^\w$ such that $|\mathcal H|<\cov(\M)$ and $P\subseteq\bigcup\mathcal H$. By the definition of the cardinal $\cov(\M)$, for some set $H\in\mathcal H$ the intersection $P\cap H$ has nonempty interior in $P$. By Claim~\ref{cl:Ho}, the set $H\supset\overline{P\cap H}$ is Haar-open in $\IR^\w$ and by Theorem \ref{Ho}, $H$ is not Haar-meager in $\IR^\w$.
\end{proof}

It remains to prove that the linear subspace $L\subseteq\IR^\w$ generated by the $G_\delta$-set $P$ is Borel and meager in $\IR^\w$.

We recall that by $h:\w^\w\to P$ we denote the homeomorphism assigning to each infinite sequence $s\in\w^\w$ the unique element $z_s$ in the intersection $\bigcap_{m\in\w}U(z_{s{\restriction}m}{\restriction}_{l_{s{\restriction}m}})$.

For every finite subset $F\subseteq\w^{<\w}$ let $l_F:=\max_{s\in F} l_s$.
For a finite subset $E\subseteq \w^\w$ let $r_E\in\w$ be the smallest number such that the restriction map $E\to\w^{r_E}$, $s\mapsto s{\restriction}_{r_E}$, is injective.
Let $l_E:=\max\{l_{s{\restriction}_{r_E}}:s\in E\}$. Taking into account that $l_{s{\restriction}_m}\ge m$ for every $s\in\w^\w$ and $m\in\w$, we conclude that $l_E\ge r_E$.

\begin{claim}\label{cl2} For every $m\in\IN$, nonempty finite set $E\subseteq \w^\w$ of cardinality $|E|\le m$,  function $\lambda:E\to [-m,m]\setminus\big[{-}\tfrac1m,\tfrac1m\big]$, the linear combination $\sum_{s\in E}\lambda(s)\cdot z_s$ belongs to the closed nowhere dense subspace $$\bigcap_{k>\max\{l_E,\Xi_m\}}\{y\in\IR^\w:|y(k)|\ge 1\}$$
of $\IR^\w$. 
\end{claim}

\begin{proof} Given any number $k>\max\{l_E,\Xi_m\}$, consider the projection $\pr_k:\IR^\w\to \IR$, $\pr_k:x\mapsto x(k)$, to the $k$-th coordinate. The claim will be proved as soon as we check that the element $y=\sum_{s\in E}\lambda(s)\cdot z_s$ satisfies $|\pr_k(y)|\ge 1$. Observe that $\pr_k(y)=\sum_{s\in F}\lambda(s)\cdot\pr_k(z_s)$. For every $s\in F$ the equality $\{z_s\}=\bigcap_{m\in\w}U(z_{s{\restriction}m}{\restriction}_{l_{s{\restriction}m}})$ implies that
$|\pr_k(z_s)-\pr_k(z_{s{\restriction}_m})|\le 2^{-l_{s{\restriction}_m}}\le 1$ for any $m\in\w$ such that $l_{s{\restriction}_m}>k$.

 Taking into account that $k>l_E=\max_{s\in E}l_{s{\restriction}_{r_E}}$, for every $s\in E$ we can find a~number $m_s\ge r_E$ such that $k\in[l_{s{\restriction}_{m_s}},l_{s{\restriction}_{m_s+1}})$.
Then
$$
\pr_k(z_{s{\restriction}_{m_s{+}1}}{\restriction}_{l_{s{\restriction}_{m_s{+}1}}})\in T_{s{\restriction}_{m_s},k}\cup\{t_{s{\restriction}_{m_s},k}\}
\subseteq T_{s{\restriction}_{m_s}}\setminus [-k,k]\subseteq T_{s{\restriction}_{m_s}}\!\setminus [-\Xi_m,\Xi_m]
$$
 by the condition $(3_{s{\restriction}_{m_s}})$ of the inductive construction.

It follows from $r_E\le\min_{s\in F}m_s$ that the sequences
$s{\restriction}_{m_s}$, $s\in E$, are pairwise distinct.
Then $$\Big|\sum_{s\in E}\lambda(s)\cdot\pr_k(z_{s{\restriction}_{m_s{+}1}}{\restriction}_{l_{s{\restriction}_{m_s{+}1}}})\Big|\ge m+1$$ by the choice of the family $(T_s)_{s\in\w^{<\w}}$. Finally,
$$
\begin{aligned}
&|\pr_k(y)|=\Big|\sum_{s\in E}\lambda(s)\cdot\pr_k(z_s)\Big|\ge\\
&\ge\Big |\sum_{s\in E}\lambda(s)\cdot\pr_k(z_{s{\restriction}_{m_s{+}1}}{\restriction}_{l_{s{\restriction}_{m_s{+}1}}})\Big|-\Big|\sum_{s\in E}\lambda(s)\cdot(\pr_k(z_{s{\restriction}_{m_s{+}1}}{\restriction}_{l_{s{\restriction}_{m_s{+}1}}})-\pr_k(z_s))\Big|\\
&\ge m+1-\sum_{s\in E}|\lambda(s)|\cdot\big|\pr_k(z_{s{\restriction}_{m_s{+}1}}{\restriction}_{l_{s{\restriction}_{m_s{+}1}}})-\pr_k(z_s)\big|\ge m+1-|E|\ge 1.
\end{aligned}
$$
\end{proof}

Claim~\ref{cl2} implies that the set $P$ is linearly independent and the linear hull $L$ of $P$ in $\IR^\w$ is meager. It remains to prove that the linear subspace $L$ is a Borel subset of $\IR^\w$. This follows from Proposition~\ref{p:hull} below.
\end{proof}

A topological space $X$ is called {\em Borel} if $X$ admits a stronger Polish topology. By the classical Lusin--Suslin Theorem \cite[15.1]{K}, a subset $B$ of a Polish space is Borel if and only if $B$ is a Borel space. 

\begin{proposition}\label{p:hull} The linear hull of any linearly independent Borel subspace of a topological vector space is Borel.
\end{proposition} 

\begin{proof} Let $B$ be a linearly independent Borel subspace of a topological vector space $E$. Being Borel, the space $B$ is the image of a Polish space $Z$ under a continuous bijective map $\xi:Z\to B$. By \cite[13.7]{K}, each Polish space admits a stronger zero-dimensional Polish topology. So we can assume that the Polish space $Z$ is zero-dimensional and hence can be identified with a $G_\delta$-subset of the real line. In this case $Z$ carries a linear order inherited from the real line. For every $n\in\w$ consider the open subset $Z^n_{<}=\big\{z\in Z^n:\forall i,j\in n\;\;\big(i<j\;\Ra\; z(i)<z(j)\big)\big\}$ in $Z^n$. Let $\IR_{\ne0}=\IR\setminus\{0\}$. By the linear independence of the set $B$ and the bijectivity of the function $\xi:Z\to B$, the continuous function
$$\Sigma_n:\IR^n_{\ne0}\times Z^n_<\to E,\quad \Sigma_n:(\lambda,z)\mapsto\sum_{i\in n}\lambda(i)\cdot\xi\big(z(i)\big),$$
is bijective. Consider the topological sum $$S=\bigoplus_{n\in\w}\big(\IR_{\ne0}^n\times Z^n_{<}\big)$$ and the continuous map 
$\Sigma:S\to E$ assigning to each pair $(\lambda,z)\in S$ the element $\Sigma_n(\lambda,z)$ where $n\in\w$ is a unique number such that $(\lambda,z)\in\IR_{\ne0}^n\times Z^n_<$. It is clear that the image $\Sigma[S]$ of $S$ under the map $\Sigma$ coincides with the linear hull $L$ of the set $S$ in $E$.  The linear independence of the set $B$ guarantees that the map $\Sigma:S\to L$ is bijective. Since the space $S$ is Polish, the space $L$ is Borel.
\end{proof}

\section{Acknowledgements} The research of
E.~Jab{\l}o\'{n}ska was partially supported
by the Faculty of Applied Mathematics AGH UST statutory tasks and dean grant within subsidy of Ministry
of Science and Higher Education.
 
\end{document}